\documentclass[11pt]{article}
\usepackage{amsfonts, amssymb, amsmath, amsthm, xcolor, graphicx, pstricks, pstricks-add}
\begin{document}

\theoremstyle{plain}
\newtheorem{thm}{Theorem}[section]
\newtheorem{cor}[thm]{Corollary}
\newtheorem{con}[thm]{Conjecture}
\newtheorem{cla}[thm]{Claim}
\newtheorem{lm}[thm]{Lemma}
\newtheorem{prop}[thm]{Proposition}
\newtheorem{example}[thm]{Example}

\theoremstyle{definition}
\newtheorem{dfn}[thm]{Definition}
\newtheorem{alg}[thm]{Algorithm}
\newtheorem{prob}[thm]{Problem}
\newtheorem{rem}[thm]{Remark}

\newcommand{\m}{\textnormal{mult}}
\newcommand{\ra}{\rightarrow}
\newcommand{\la}{\leftarrow}

\newcommand{\A}{\mathcal{A}}
\newcommand{\B}{\mathcal{B}}
\newcommand{\D}{\mathcal{D}}
\newcommand{\E}{\mathcal{E}}
\newcommand{\F}{\mathcal{F}}

\renewcommand{\baselinestretch}{1.1}

\title{\bf An Analogue of Hilton-Milner Theorem \\ for Set
Partitions}
\author{
Cheng Yeaw Ku
\thanks{ Department of Mathematics, National University of
Singapore, Singapore 117543. E-mail: matkcy@nus.edu.sg} \and Kok
Bin Wong \thanks{
Institute of Mathematical Sciences, University of Malaya, 50603
Kuala Lumpur, Malaysia. E-mail:
kbwong@um.edu.my.} } \maketitle

\begin{abstract}\noindent
Let $\mathcal{B}(n)$ denote the collection of all set partitions
of $[n]$. Suppose $\mathcal{A} \subseteq \mathcal{B}(n)$ is a
non-trivial $t$-intersecting family of set partitions i.e. any two members of $\A$ have at least $t$ blocks in common, but there is no fixed $t$ blocks of size one which belong to all of them.
It is proved that for sufficiently large $n$ depending on $t$,
\[ |\mathcal{A}| \le B_{n-t}-\tilde{B}_{n-t}-\tilde{B}_{n-t-1}+t
\]
where $B_{n}$ is the $n$-th Bell number and $\tilde{B}_{n}$ is
the number of set partitions of $[n]$ without blocks of size
one. Moreover, equality holds if and only if $\mathcal{A}$ is
equivalent to
\[ \{ P \in \mathcal{B}(n): \{1\}, \{2\},\dots, \{t\}, \{i\} \in P~~\textnormal{for
some}~i \not = 1,2,\dots, t,n \}\cup \{Q(i,n)\ :\ 1\leq i\leq t\} \]
where $Q(i,n)=\{\{i,n\}\}\cup\{\{j\}\ :\ j\in [n]\setminus \{i,n\}\}$. This is an
analogue of the Hilton-Milner theorem for set partitions.
\end{abstract}

\bigskip\noindent
{\sc keywords:} intersecting family, Hilton-Milner, Erd{\H
o}s-Ko-Rado, set partitions

\section{Introduction}

\subsection{Finite sets}

Let $[n]=\{1, \ldots, n\}$ and ${[n] \choose k}$ denote the
family of all $k$-subsets of $[n]$.

One of the most beautiful result in extremal combinatorics is
the Erd{\H o}s-Ko-Rado theorem (\cite{EKR}, \cite{Frankl}, \cite{Wilson}) which asserts that if a
family $\mathcal{A} \subseteq {[n] \choose k}$ is $t$-{\em
intersecting} (i.e. $\vert A\cap  B\vert\geq t$ for any $A,B\in \A$) and $n >2k-t$, then $|\mathcal{A}|
\le {n-t \choose k-t}$ for $n\geq (k-t+1)(t+1)$. 

\begin{thm}[Erd{\H o}s, Ko, and Rado \cite{EKR}, Frankl \cite{Frankl}, Wilson \cite{Wilson}]\label{EKR} Suppose $\mathcal{A} \subseteq {[n] \choose k}$ is $t$-intersecting and $n>2k-t$. Then for $n\geq (k-t+1)(t+1)$,
\begin{equation}
\vert \mathcal{A} \vert\leq {n-t \choose k-t}.\notag
\end{equation}
Moreover, if $n>(k-t+1)(t+1)$, equality holds if and only if $\mathcal{A}=\{A\in {[n] \choose k}\ :\ T\subseteq A\}$ for some $t$-set $T$.
\end{thm}

For a family $\mathcal{A}$ of $k$-subsets, $\mathcal{A}$ is said
to be {\em trivially} $t$-intersecting if there exists a $t$-set $T=\{x_1,\dots, x_t\}$ such that all members of $\mathcal{A}$ contains $T$. The
Erd{\H o}s-Ko-Rado theorem implies that a $t$-intersecting family
of maximum size must be trivially $t$-intersecting when $n$ is sufficiently large in terms of $k$ and $t$.

Hilton and Milner \cite{HM} proved a strengthening of the Erd{\H
o}s-Ko-Rado theorem for $t=1$ by determining the maximum size of a
non-trivial 1-intersecting family. A short and elegant proof was
later given by Frankl and F{\" u}redi \cite{FF} using the
shifting technique.

\begin{thm}[Hilton-Milner]\label{Hilton-Milner}
Let $\mathcal{A} \subseteq {[n] \choose k}$ be a non-trivial
1-intersecting family with $k \ge 4$ and $n>2k$. Then
\[ |\mathcal{A}| \le {n-1 \choose k-1} - {n-k-1 \choose k-1}+1.
\]
Equality holds if and only if
\[ \mathcal{A} = \{X \in {[n] \choose k}: x \in X, X \cap Y \not
= \emptyset\} \cup \{Y\}\]
for some $k$-subset $Y \in {[n] \choose k}$ and $x \in X
\setminus Y$.
\end{thm}

\subsection{Permutations and set partitions}

The main result of this paper is motivated by recent investigations of the Erd{\H o}s-Ko-Rado type of problems for permutations and set partitions.

The study of intersecting families of permutations was initiated by Deza and Frankl \cite{DF} in the context of coding theory. Let $\mathrm{Sym}(n)$ denote the set of all permutations of $[n]$. A family $\mathcal{A} \subseteq \mathrm{Sym}(n)$ is $t$-{\em intersecting} if $\left|\{x: g(x)=h(x)\}\right| \ge t$ for any $g, h \in \A$. 

Recently, Ellis, Friedgut and Pilpel \cite{EFP} showed that for sufficiently large $n$ depending on $t$, a $t$-intersecting family $\A$ of permutations has size at most $(n-t)!$, with equality if and only if $\A$ is a coset of the stabilizer of $t$ points, thus settling an old conjecture of Deza and Frankl in the affirmative. The proof uses spectral methods and representations of the symmetric group. Subsequently, building on the representation theorectic approach, Ellis \cite{E} proved an analogue of the Hilton-Milner theorem  for $t$-intersecting families of permutations. The readers may also refer to \cite{CK, GM, HK, KL, KW, LM, Toku, WZ} for some recent results on the Erd{\H o}s-Ko-Rado type of problems.

On the other hand, recall that a set partition of $[n]$ is a collection of pairwise
disjoint nonempty subsets (called {\em blocks}) of $[n]$ whose
union is $[n]$. Let $\mathcal{B}(n)$ denote the family of all
set partitions of $[n]$. It is well-known that the size of
$\mathcal{B}(n)$ is the $n$-th Bell number, denoted by $B_{n}$.
A block of size one is also known as a {\em singleton}. We
denote the number of all set partitions of $[n]$ which are
singleton-free (i.e. without any singleton) by $\tilde{B}_{n}$.

A family
$\mathcal{A} \subseteq \mathcal{B}(n)$ is said to be $t$-{\em
intersecting} if any two of its members have at least $t$ blocks
in
common. It is {\em trivially} $t$-intersecting if it consists of set
partitions containing $t$ fixed singletons. Note that if it is trivially $t$-intersecting of maximum size then it consists of all set partitions containing the $t$ fixed singletons.

Motivated by the Erd{\H o}s-Ko-Rado theorem, Ku and Renshaw
\cite[Theorem 1.7 and Theorem 1.8]{KR} proved the following analogue of the Erd{\H o}s-Ko-Rado
theorem for set partitions.

\begin{thm}[Ku-Renshaw]\label{Ku-Renshaw}
Suppose $\mathcal{A} \subseteq \mathcal{B}(n)$ is
a $t$-intersecting family. Then for  $n\geq n_0(t)$, 
\[ |\mathcal{A}| \le B_{n-t}, \]
with equality if and only if $\mathcal{A}$ is a trivially
$t$-intersecting family of maximum size.
\end{thm}

In view of the Hilton-Milner theorem, the aim of this paper is to determine the size and the structure of non-trivial $t$-intersecting families of set partitions. The analogy to set systems and permutations suggests that almost all members of
such a family should share $t$ common singletons. The following is
an example of a large non-trivial $t$-intersecting family.

Let $a_1,a_2,\dots, a_t, b \in [n]$, and all the $a_i$'s and $b$ are distinct. Let $Q(a_i,b) \in \B(n)$ be the set partition containing $\{a_i,b\}$ and $\{j\}$ for all $j \not = a_i,b$. Set $\mathcal Q=\{Q(a_1,b)\}\cup\cdots\cup \{Q(a_t,b)\}$. A Hilton-Milner type family is
given by
\[ \mathcal{H}(a_1,\dots,a_t,b) = \{ P \in \B(n): \{a_1\},\dots, \{a_t\}, \{c\} \in P
~\textrm{for some}~ c \not = a_1,\dots, a_t,b\} \cup \mathcal Q .\]

It is easily verified that
\[ |\mathcal{H}(a_1,\dots,a_t,b)| =
B_{n-t}-\tilde{B}_{n-t}-\tilde{B}_{n-t-1}+t.\]
Indeed, the number of set partitions having $\{a_1\}$, $\{a_2\}$, $\dots$, $\{a_t\}$ as the only
singletons is $\tilde{B}_{n-t}$, and the number of set partitions
having $\{a_1\}$, $\{a_2\}$, $\dots$, $\{a_t\}$  and $\{b\}$ as the only singletons is
$\tilde{B}_{n-t-1}$.

Using an analogue of the shifting operation for set partitions
(called the {\em splitting} operation) first introduced by Ku
and Renshaw in \cite{KR}, we prove the following analogue of the 
Hilton-Milner theorem.

\begin{thm}\label{main}
Suppose $\mathcal{A}$ is a non-trivial $t$-intersecting family of
set partitions of $[n]$. Then, for $n\geq n_0(t)$,
\[ |\mathcal{A}| \le
B_{n-t}-\tilde{B}_{n-t}-\tilde{B}_{n-t-1}+t,\]
with equality if and only if $\mathcal{A} = \mathcal{H}(a_1,\dots,a_t,b)$
for some $a_1,\dots, a_t, b \in [n]$.
\end{thm}

\section{Splitting operation}

In this section, we summarize some important results regarding
the splitting operation for intersecting family of set
partitions. We refer the reader to \cite{KR} for proofs which
are omitted here.

Let $i, j \in [n]$, $i \not = j$, and $P \in \mathcal{B}(n)$.
Denote by
$P_{[i]}$ the block of $P$ which contains $i$. We define the
$(i,j)$-{\it
split} of $P$ to be the following set partition:
\[ s_{ij}(P) = \left\{\begin{array}{ll}
P \setminus \{P_{[i]}\} \cup \{\{i\}, P_{[i]}\setminus\{i\}\} &
\textrm{if } j \in P_{[i]}, \\
P      &   \textrm{otherwise. }
\end{array} \right.   \]
For a family $\mathcal{A} \subseteq \B(n)$, let
$s_{ij}(\mathcal{A}) = \{ s_{ij}(P):
P \in \mathcal{A}\}$. Any family $\mathcal{A}$ of set partitions
can be decomposed
with respect to given $i$, $j \in [n]$ as follows:
\begin{eqnarray}
\A & = & (\A \setminus \A_{ij}) \cup \A_{ij}, \nonumber
\end{eqnarray}
where $\A_{ij} = \{ P \in \A: s_{ij}(P) \not \in \A\}$. Define
the
$(i,j)$-{\em splitting} of $\A$ to be the family
\[ S_{ij}(\A) = (\A \setminus \A_{ij}) \cup s_{ij}(\A_{ij}). \]

Let $I(n,t)$ denote the set of all $t$-intersecting families of
set partitions of $[n]$. Surprisingly, it turns out that for any
$\A \in I(n,t)$, splitting operations preserve the size and the
intersecting property.

\begin{prop}[\cite{KR}, Proposition 3.2] \label{P1}
 Let $\A \in I(n,t)$. Then
$S_{ij}(\A) \in I(n,t)$ and $|S_{ij}(\A)|=|\A|$.
\end{prop}

A family $\A$ of set partitions is {\it compressed} if for any
$i,j
\in [n]$, $i \not = j$, we have $S_{ij}(\A)=\A$. For a set
partition
$P$, let $\sigma(P)=\{ x : \{x\} \in P\}$ denote the union of
its
singletons (block of size $1$). For a family $\A$ of set
partitions,
let $\sigma(\A)=\{\sigma(P): P \in \A\}$. Note that $\sigma(\A)$
is a family of subsets of $[n]$.

\begin{prop}[\cite{KR}, Proposition 3.3] \label{P2}
Given a family $\A \in I(n,t)$, by repeatedly applying the splitting
operations, we eventually obtain a compressed family $\A^{*} \in I(n,t)$ with $|\A^{*}|=|\A|$.
\end{prop}

For a compressed family $\A$, its intersecting property can be
transferred to $\sigma(\A)$, thus allowing us to access the
structure of $\A$ via the structure of $\sigma(\A)$.

\begin{prop}[\cite{KR}, Proposition 3.4] \label{P3}
If $\A \in I(n,t)$ is compressed, then $\sigma(\A)$ is a
$t$-intersecting family of subsets of $[n]$.
\end{prop}

\begin{lm} \label{lm_ij}  Suppose $\A \in I(n,t)$ and $S_{ij}(\mathcal{A}) = \mathcal{H}(a_1,\dots ,a_t,b)$.  If 
\begin{itemize}
\item[\textnormal{(a)}] 
\begin{equation}
P_e=\{ \{a_1\}, \dots,\{a_t\}, \{e\}, [n]\setminus \{a_1,\dots, a_t,e\}\}\in \A,\notag\\
\end{equation}
for all $e\in [n]\setminus \{a_1,\dots, a_t,b\}$, and
\item[\textnormal{(b)}] $Q(a_l,b)\in \A$ for all $1\leq l\leq t$,
\end{itemize}
then $\A= \mathcal{H}(a_1,\dots ,a_t,b)$.
\end{lm}

\begin{proof} Suppose there is a $P\in S_{ij}(\A)\setminus \A$. Then $P=s_{ij}(T)$ for some $T\in\A$ and $T\notin S_{ij}(\A)$. 

\noindent
{\bf Case 1.} Suppose $i\neq a_1,\dots, a_t$. If $i\neq b$, then $P$ consists of $\{a_1\}$, $\dots$, $\{a_t\}$, $\{i\}$ and $B$, where $B$ is a set partition of $[n]\setminus\{a_1,\dots, a_t, i\}$. Suppose $B$ does not contain any singleton. 
If $j\neq a_1,\dots, a_t$, then the only singletons in $T$ are $\{a_1\}$, $\dots$, $\{a_t\}$. If $j=a_{l_1}$ for some $1\leq l_1\leq t$, then the only singletons in $T$ are $\{a_1\}$, $\dots$, $\{a_{l_1-1}\}$, $\{a_{l_1+1}\}$, $\dots$, $\{a_{t}\}$. In all cases, $T$ has no singletons other than  $\{a_1\}$, $\dots$, $\{a_t\}$.
Therefore $\vert T\cap Q(a_1,b)\vert\leq t-1$, contradicting the fact that $\A$ is $t$-intersecting. Similarly if $B$ contains the singleton $\{b\}$ or $\{j\}$ only, then $\vert T\cap Q(a_1,b)\vert\leq t-1$, a contradiction. Hence $B$ contains a singleton $\{e\}$ for some $e\neq b,j$. This means that $T$ contains the singletons $\{a_1\}$, $\dots$, $\{a_t\}$, $\{e\}$, and so $T \in \mathcal{H}(a_{1}, \ldots, a_{t}, b)$, contradicting the fact that $T \not \in S_{ij}(\A)$. 

If $i=b$, then $P$ consists of $\{a_1\}$, $\dots$, $\{a_t\}$, $\{b\}$, $\{e_1\}$ and $B$, where 
$e_1\in [n]\setminus \{a_1,\dots, a_t,b\}$ and  $B$ is a set partition of $[n]\setminus\{a_1,\dots, a_t, b, e_1\}$.
If $j\neq a_1,\dots, a_t$, then $T$ contains the singletons $\{a_1\}$, $\dots$, $\{a_t\}$, $\{e_1\}$, a contradiction. Suppose $j=a_{l_0}$ for some $1\leq l_0\leq t$. Suppose $B$ contains a block of size at least 2. Let $e_0$ be an element in this block. Note that $\vert P_{e_0}\cap T\vert=t-1$, a contradiction. So we may assume that $B$ consists of singletons, but then $T=Q(a_{l_0},b)$, a contradiction.

\noindent
{\bf Case 2.} Suppose $i=a_{l_0}$ for some $1\leq l_0\leq t$. Then $P$ consists of $\{a_1\}$, $\dots$, $\{a_t\}$, $\{e\}$ and $B$, where 
$e\in [n]\setminus \{a_1,\dots, a_t,b\}$ and  $B$ is a set partition of $[n]\setminus\{a_1,\dots, a_t, e\}$. Suppose $B$ contains a block of size at least 2. Let $e_0$ be an element in this block. We may assume $e_0\neq b$. Note that $\vert P_{e_0}\cap T\vert=t-1$, a contradiction. So we may assume that $B$ consists of singletons, but then $T=Q(a_{l_0},b)$, a contradiction.

Hence $S_{ij}(\A)=\A$.
\end{proof}

The following proposition says that a Hilton-Milner type family
is preserved when `undoing' the splitting operations.

\begin{prop} \label{undo}  Suppose $n\geq t+3$, $\A \in I(n,t)$ and  $S_{ij}(\mathcal{A}) = \mathcal{H}(a_1,\dots ,a_t,b)$. Then $\mathcal{A} = \mathcal{H}(a_1,\dots ,a_t,b)$.
\end{prop}

\begin{proof} It is sufficient to show that conditions (a) and (b) of Lemma \ref{lm_ij} hold.

Let $e\in [n]\setminus \{a_1,\dots, a_t,b\}$ and 
\begin{equation}
P_e=\{ \{a_1\}, \dots,\{a_t\}, \{e\}, [n]\setminus \{a_1,\dots, a_t,e\}\}.\notag\\
\end{equation}
Note that $P_e\in S_{ij}(\A)$ and $\vert [n]\setminus \{a_1,\dots, a_t,e\}\vert\geq 2$. 

\noindent
{\bf Case 1.} Suppose $i,j\neq a_1,\dots, a_t$.  Assume that $e\neq i$. If $P_e\notin \A$, then $P_e=s_{ij}(T_e)$ for some $T_e\in\A$, a contradiction, for $i$ cannot be contained in a block of size greater than 1 after the splitting operation. So $P_e\in \A$ for all $e\in [n]\setminus \{a_1,\dots, a_t,b, i\}$. 

Suppose $j\neq b$. Now if $Q(a_l,b)\notin \A$ for some $1\leq l\leq t$, then 
\begin{align}
W_1 & =\{ \{a_l,b\}\}\cup \{ \{i,j\}\}\cup \{ \{q\}\ :\ q\in [n]\setminus \{a_l, b, i, j\}\}\in \A,\notag
\end{align}
and $\vert W_1\cap P_{j}\vert =t-1$, contradicting the fact that $\A$ is $t$-intersecting. So $Q(a_l,b)\in \A$ for all $1\leq l\leq t$. It remains to show that $P_i\in \A$ when $i\neq b$. If $P_i\notin \A$, then 
\begin{equation}
W_2=\{ \{a_1\}, \dots,\{a_t\}, [n]\setminus \{a_1,\dots, a_t\}\}\in \A,\notag
\end{equation}
a contradiction, for $\vert W_2\cap Q(a_1,b)\vert=t-1$. 

Suppose  $j=b$. Then $i\neq b$.
If $P_i\notin \A$, then $W_2\in \A$. If $Q(a_l,b)\notin \A$ for some $1\leq l\leq t$, then 
\begin{align}
W_3 & =\{ \{a_l,b,i\}\}\cup \{ \{q\}\ :\ q\in [n]\setminus \{a_l, b, i\}\}\in\A,\notag
\end{align}
and  $\vert W_2\cap W_3\vert=t-1$, a contradiction. If $Q(a_l,b)\in \A$ for some $1\leq l\leq t$, then 
 $\vert W_2\cap Q(a_l,b)\vert=t-1$, again a contradiction. Hence $P_i\in \A$. Since $\vert W_3\cap P_i\vert=t-1$, we conclude that $Q(a_l,b)\in \A$ for all $1\leq l\leq t$. 
 
\noindent
{\bf Case 2.} Suppose $i=a_{l_0}$ and $j=a_{l_1}$ for some $1\leq l_0, l_1\leq t$.  Without loss of generality assume that $l_0=1$ and $l_1=2$. Note that $Q(a_{1},b)\in \A$. Now if $P_e\notin \A$, then 
\begin{align}
W_1 & =\{ \{a_1,a_2\},\{a_3\}, \dots, \{a_t\}, \{e\}, [n]\setminus \{a_1,\dots, a_t,e\}\}\in\A,\notag
\end{align}
a contradiction, for $\vert Q(a_1,b)\cap W_1\vert=t-1$. So $P_e\in \A$ for all $e\in [n]\setminus \{a_1,\dots, a_t,b\}$.

Now if $Q(a_l,b)\notin \A$ for some $3\leq l\leq t$, then 
\begin{equation}
W_2=\{\{a_l,b\}\}\cup \{\{a_1,a_2\}\}\cup \{\{q\}\ :\ q\in [n]\setminus \{a_1,a_2,a_l,b\}\}\in \A,\notag
\end{equation}
a contradiction, as $\vert P_e\cap W_2\vert=t-2$. Next if $Q(a_2,b)\notin \A$, then 
\begin{equation}
W_3=\{\{a_1,a_2,b\}\}\cup \{\{q\}\ :\ q\in [n]\setminus \{a_1,a_2,b\}\}\in \A,\notag
\end{equation}
again a contradiction, as $\vert P_e\cap W_3\vert=t-1$. Thus $Q(a_l,b)\in \A$ for all $1\leq l\leq t$.

\noindent
{\bf Case 3.} Suppose $j=a_{l_0}$ for some $1\leq l_0\leq t$. Without loss of generality assume that $l_0=1$. 
As in Case 1, $P_e\in \A$ for all $e\in [n]\setminus \{a_1,\dots, a_t,b, i\}$. By Case 2, we may assume that $i\neq a_1,\dots, a_t$. 

Suppose $i=b$. Then $Q(a_{l},b)\in\A$ for all $1\leq l\leq t$, and we are done.

Suppose $i\neq b$. If $Q(a_1,b)\notin \A$, then 
\begin{equation}
W=\{\{a_1,b,i\}\}\cup \{ \{q\}\ :\ q\in [n]\setminus \{a_1,b,i\}\}\in \A.\notag
\end{equation}
 Since $P_i\in S_{ia_1}(\A)$, 
  and $\vert P_i\cap W\vert =t-1$, we must have 
\begin{equation}
R'=\{\{a_1,i\}, \{a_2\},\dots, \{a_t\}, [n]\setminus\{a_1,\dots, a_t,i\}\}\in \A,\notag
\end{equation}  
 but then $\vert R'\cap W\vert =t-1$, a contradiction. Hence $Q(a_1,b)\in \A$. Now  $\vert Q(a_1,b)\cap R'\vert =t-1$ implies that $P_i\in \A$. 
Next if $Q(a_l,b)\notin \A$ for some $2\leq l\leq t$, then 
\begin{equation}
W_2=\{\{a_l,b\}\}\cup \{\{a_1,i\}\}\cup \{\{q\}\ :\ q\in [n]\setminus \{a_1,i,a_l,b\}\}\in \A,\notag
\end{equation}
a contradiction, as $\vert P_i\cap W_2\vert=t-2$. Thus $Q(a_l,b)\in \A$ for all $1\leq l\leq t$.

\noindent
{\bf Case 4.} Suppose $i=a_{l_0}$ for some $1\leq l_0\leq t$. Without loss of generality assume that $l_0=1$.   
By Case 2, we may assume that $j\neq a_1,\dots, a_t$. 

Suppose $j=b$. Note that $Q(a_{1},b)\in\A$. Let $e_0,e_1\in [n]\setminus \{a_1,\dots, a_t,b\}$, $e_0\neq e_1$. If both $P_{e_0}$ and $P_{e_1}$ are not contained in $\A$, then $W_0, W_1\in \A$, where
\begin{align}
W_0 &=\{ \{a_2\}, \dots,\{a_t\}, \{e_0\}, [n]\setminus \{a_2,\dots, a_t,e_0\}\},\notag\\
W_1 &=\{ \{a_2\}, \dots,\{a_t\}, \{e_1\}, [n]\setminus \{a_2,\dots, a_t,e_1\}\}.\notag
\end{align}
We have obtained a contradiction, as $\vert W_0\cap W_1\vert=t-1$. So we may assume $P_{e_{0}}\in\A$. If $Q(a_l,b)\notin \A$ for some $2\leq l\leq t$, then 
\begin{align}
W_2 & =\{ \{a_l,b,a_1\}\}\cup \{ \{q\}\ :\ q\in [n]\setminus \{a_l, b, a_1\}\}\in\A,\notag
\end{align}
and  $\vert W_2\cap P_{e_0}\vert=t-1$, a contradiction. Thus $Q(a_{l},b)\in\A$ for all $1\leq l\leq t$. Since $\vert Q(a_2,b)\cap W_1\vert=t-1$, we conclude that $P_{e_1}\in \A$. In fact, $P_e\in \A$ for all $e\in [n]\setminus \{a_1,\dots, a_t,b\}$.

Suppose $j\neq b$. Note that $Q(a_{1},b)\in\A$. If $P_j\notin \A$, then 
\begin{align}
W_3 & =\{ \{a_1,j\}, \{a_2\},\dots, \{a_t\}, [n]\setminus \{a_1,\dots, a_t,j\}\}\in\A.\notag
\end{align}
This contradicts that $\A$ is $t$-intersecting as $\vert Q(a_1,b)\cap W_3\vert=t-1$. Thus $P_j\in \A$.
Now if $Q(a_l,b)\notin \A$ for some $2\leq l\leq t$, then 
\begin{equation}
W_4=\{\{a_l,b\}\}\cup \{\{a_1,j\}\}\cup \{\{q\}\ :\ q\in [n]\setminus \{a_1,j,a_l,b\}\}\in \A.\notag
\end{equation}
This contradicts that $\A$ is $t$-intersecting as $\vert P_j\cap W_4\vert=t-2$. Thus $Q(a_l,b)\in \A$ for all $1\leq l\leq t$.
Finally if $P_e\notin \A$ and $e\neq j$, then 
\begin{align}
W_5 &=\{ \{a_2\}, \dots,\{a_t\}, \{e\}, [n]\setminus \{a_1,a_2,\dots, a_t,e\}\}\in \A,\notag
\end{align}
a contradiction, for $\vert Q(a_2,b)\cap W_5\vert=t-1$. Hence $P_e\in \A$ for all $e\in [n]\setminus \{a_1,\dots, a_t,b\}$.
\end{proof}

\section{Proof of main result}

The following identities for $B_{n}$ and $\tilde{B}_{n}$ are straightforward.

\begin{lm}\label{identity}
Let $n \ge 2$. Then
\begin{eqnarray}
B_{n} & = & \sum_{k=0}^{n} {n \choose k} \tilde{B}_{n-k},
\label{e1} \\
\tilde{B}_{n} & = & \sum_{k=1}^{n-1} {n-1 \choose k}
\tilde{B}_{n-1-k}, \label{e2}
\end{eqnarray}
with the conventions $B_{0}=\tilde{B}_{0}=1$.
\end{lm}
Note in passing that $\tilde{B}_{1}=0$. By (\ref{e1}) and (\ref{e2}),
\begin{equation}
B_n=\tilde{B}_n+\tilde{B}_{n+1}=B_{n-1}-\tilde{B}_{n-1}+\tilde{B}_{n+1}\leq B_{n-1}+\tilde{B}_{n+1}.\label{ttt}
\end{equation}
 Since $\lim_{n\to \infty} B_{n}/B_{n-1}=\infty$ (see \cite[Corollary 2.7]{Klazar}), we deduce that 
\begin{equation}
\lim_{n\to \infty} \tilde{B}_{n+1}/B_{n-1}=\infty.\label{inequal}
\end{equation}
 Next note that $(B_{n-1}-\tilde{B}_{n-1}-\tilde{B}_{n-2})/B_{n-2}=(\tilde{B}_{n}-\tilde{B}_{n-2})/B_{n-2}\geq \tilde{B}_{n}/B_{n-2}-1$. So $\lim_{n\to \infty} (B_{n-1}-\tilde{B}_{n-1}-\tilde{B}_{n-2})/B_{n-2}=\infty$ and Lemma \ref{less} follows. Lemma \ref{less02} follows by noting that $B_{n-r+1}\leq B_{n-t-3}$.

\begin{lm}\label{less}
Let $c$ be a fixed positive integer. Then, for $n\geq n_0(t)$,
\[ cB_{n-t-1} < B_{n-t}-\tilde{B}_{n-t}-\tilde{B}_{n-t-1}. \]
\end{lm}

\begin{lm}\label{less02}
If $t+4 \le r \le n-2$ and $n\geq n_0(t)$, then
\[ tB_{n-r+1} < \tilde{B}_{n-t-1}. \]
\end{lm}

\begin{lm}\label{less03} For  $n\geq n_0(t)$,
\[ \tilde{B}_{n-t-1}> \sum_{k=\left \lfloor \frac{n}{t+1}+t-1 \right \rfloor + 1}^{n} {n \choose k}\tilde{B}_{n-k}. \]
\end{lm}

\begin{proof} By  (\ref{e2}), 
\begin{align}
\sum_{k=\left \lfloor \frac{n}{t+1}+t-1 \right \rfloor + 1}^{n} {n \choose k}\tilde{B}_{n-k} & \leq \tilde{B}_{n-\left \lfloor \frac{n}{t+1}+t-1 \right \rfloor+1} \sum_{k=\lfloor n/(t+1)+t-1 \rfloor + 1}^{n} {n \choose k}\notag\\
&\leq  2^n\tilde{B}_{n-\left \lfloor \frac{n}{t+1}+t-1 \right \rfloor+1}.\notag
\end{align}
 So it is sufficient to show that $\tilde{B}_{n-t-1}/\tilde{B}_{n-\left \lfloor \frac{n}{t+1}+t-1 \right \rfloor+1}>2^n$.

Again by (\ref{e2}), for any fixed $r$, $\tilde{B}_{m}/\tilde{B}_{m-2}>r$ for sufficiently large $m$. Therefore
\begin{align}
\frac{\tilde{B}_{n-t-1}}{\tilde{B}_{n-\left \lfloor \frac{n}{t+1}+t-1 \right \rfloor+1}}&\geq \left( \frac{\tilde{B}_{n-\left \lfloor \frac{n}{t+1}+t-1 \right \rfloor+2u-1}}{\tilde{B}_{n-\left \lfloor \frac{n}{t+1}+t-1 \right \rfloor+2u-3}}\right)\cdots \left( \frac{\tilde{B}_{n-\left \lfloor \frac{n}{t+1}+t-1 \right \rfloor+5}}{\tilde{B}_{n-\left \lfloor \frac{n}{t+1}+t-1 \right \rfloor+3}}\right)\left( \frac{\tilde{B}_{n-\left \lfloor \frac{n}{t+1}+t-1 \right \rfloor+3}}{\tilde{B}_{n-\left \lfloor \frac{n}{t+1}+t-1 \right \rfloor+1}}\right)\notag\\
& > r^{u-1},\notag
\end{align}
where $u=\lfloor \frac{1}{2}(\lfloor \frac{n}{t+1}+t-1  \rfloor-t-2)\rfloor$. Clearly $u-1\geq \frac{n}{4(t+1)}$. So if we choose $r=2^{4(t+1)}$, then for sufficiently large $n$, the lemma follows.
\end{proof}

\begin{lm}\label{Pre_Case_1} Let $\mathcal{A}$ be a non-trivial $t$-intersecting family of set
partitions of $[n]$ of maximum size. Suppose for all $i, j
\in [n]$ such that $S_{ij}(\A)\neq \A$, $S_{ij}(\A)$ is trivially $t$-intersecting. If $S_{ab}(\A)\neq \A$ for some $a, b
\in [n]$, then for $n\geq n_0(t)$, we have
\[ \A = \A_{1} \cup \A_{2},  \]
and either \textnormal{(\ref{poss1})} or \textnormal{(\ref{poss2})} holds:
\begin{align}
\A_{1}  &\subseteq \{C \in \B(n): \{a\}, \{b\}\in
C\},\notag\\
 \emptyset \not = \A_{2} & \subseteq \{C \in \B(n):
\{a, b\}\in C\},\label{poss1}\\
\A &\subseteq \{C \in \B(n): \{y_3\}, \dots, \{y_t\} \in
C\},\notag
\end{align}
for some fixed $y_3,\dots, y_t\in [n]\setminus \{a,b\}$, or 
\begin{align}
\A_{1}  &\subseteq \{C \in \B(n): \{a\} \in
C\},\notag\\
 \emptyset \not = \A_{2} & \subseteq \{C \in \B(n):
\{a, b\}\in C\},\label{poss2}\\
\A &\subseteq \{C \in \B(n): \{x_2\}, \dots, \{x_t\} \in
C\},\notag
\end{align}
for some fixed $x_2,\dots, x_t\in [n]\setminus \{a,b\}$. Here, $y_{3}$, $\ldots$, $y_{t}$ only exist if $t \ge 3$ and  $x_{2}$, $\ldots$, $x_{t}$ only exist if $t \ge 2$.
\end{lm}

\begin{proof} By assumption, $S_{ab}(\A)$ is trivially $t$-intersecting. 
This means that either 
\begin{itemize}
\item[(a)] $\{a\}, \{b\}, \{y_3\},\dots, \{y_t\}\in P$ for all $P\in S_{ab}(\A)$, or
\item[(b)] $\{a\}, \{x_2\}, \dots, \{x_t\}\in P$ for all $P\in S_{ab}(\A)$, or
\item[(c)] $\{b\}, \{x_2\}, \dots, \{x_t\}\in P$ for all $P\in S_{ab}(\A)$.
\end{itemize}

Suppose (a) holds. Since $\A$ is non-trivially $t$-intersecting, we conclude that (5) holds.

Suppose (b) or (c) holds. Since $\A$ is non-trivially $t$-intersecting, there is a $P_0\in \A$ such that $s_{ab}(P_0)\notin \A$  and $P_0=\{\{a,b\}\cup X_1\}\cup \{ \{x_l\}\ : \ 2\leq l\leq t\}\cup B_1$ where $X_1\subseteq [n]\setminus \{a,b, x_2,\dots, x_t\}$ and $B_1$ is a set partition of $[n]\setminus (\{a,b, x_2,\dots, x_t\} \cup X_1)$ (we allow $X_1=\varnothing$). 

Suppose there is a $Q\in \A$ such that $Q=\{\{a,b\} \cup X_2\}\cup \{ \{x_l\}\ : \ 2\leq l\leq t\}\cup B_2$ where $X_2\subseteq [n]\setminus \{a,b,x_2,\dots, x_t\}$ and $B_2$ is a set partition of $[n]\setminus (\{a,b, x_2,\dots, x_t\} \cup X_2)$. Suppose $X_2\nsubseteq  X_1$. Let $d\in X_2\setminus X_1$. If $s_{ad}(Q), s_{bd}(Q)\in \A$, then $s_{ab}(s_{bd}(Q))=s_{bd}(Q)$ and $s_{ab}(s_{ad}(Q))=s_{ad}(Q)$. But this contradicts (b) and (c), as $\{a\}$ is not a block in $s_{bd}(Q)$ and $\{b\}$ is not a block in $s_{ad}(Q)$.
So we may assume $s_{bd}(Q)\notin \A$. 
Since $s_{bd}(P_0)=P_0$, we see that $S_{bd}(\A)$ is non-trivially $t$-intersecting and $S_{bd}(\A)\neq \A$, a contradiction. So we may assume $X_2\subseteq X_1$. 
 If there is a $c\in X_1\setminus X_2$, then $s_{ac}(P_0)=s_{ab}(P_0)$, $s_{ac}(Q)=Q$, and thus $S_{ac}(\A)$ is non-trivially $t$-intersecting and $S_{ac}(\A)\neq \A$, a contradiction.  Therefore we may assume that
\[ \A = \A_{1} \cup \A_{2},  \]
where 
\begin{align}
\A_{1}  &\subseteq \{C \in \B(n): \{a\} \in
C\},\notag\\
 \emptyset \not = \A_{2} & \subseteq \{C \in \B(n):
\{a, b\} \cup  X_1\in C\},\notag\\
\A &\subseteq \{C \in \B(n): \{x_2\}, \dots, \{x_t\} \in
C\}.\notag
\end{align}

Suppose $X_1\neq\varnothing$. This implies that (b) holds. Note that $s_{ba}(P_0)\notin \A$, for otherwise $s_{ab}(s_{ba}(P_0))=s_{ba}(P_0)\in S_{ab}(\A)$ and it does not contain the singleton $\{a\}$. Now $S_{ba}(\A)\neq \A$ implies that $S_{ba}(\A)$ is trivially $t$-intersecting (by assumption). Furthermore every element in $S_{ba}(\A)$ contains the singleton $\{b\}$. Since $S_{ba}(\A_1)=\A_1$, we must have $\A_{1} \subseteq \{C \in \B(n): \{a\}, \{b\}, \{x_2\},\dots, \{x_t\}\in C\}$. Therefore $\vert \A_1\vert\leq B_{n-t-1}$, $\vert \A_2\vert\leq B_{n-t-1}$ and
$\vert \A\vert\leq 2B_{n-t-1}< B_{n-t}-\tilde{B}_{n-t}-\tilde{B}_{n-t-1}$ (Lemma \ref{less}), a contradiction, as $\mathcal{A}$ is a non-trivial $t$-intersecting family of maximum size. 
\end{proof}

\begin{thm}\label{Case_1} Let $\mathcal{A}$ be a non-trivial $t$-intersecting family of set
partitions of $[n]$ of maximum size. If $\mathcal{A}$ is not compressed, then  for
$n\geq n_0(t)$, there exist $k, l
\in [n]$ such that $S_{kl}(\A)\neq \A$ and  $S_{kl}(\A)$ is non-trivially $t$-intersecting.
\end{thm}

\begin{proof} Assume, for a contradiction, that for all $i, j
\in [n]$ such that $S_{ij}(\A)\neq \A$, $S_{ij}(\A)$ is trivially $t$-intersecting. 

Since $\mathcal{A}$ is not compressed, there exist $a, b
\in [n]$ with $S_{ab}(\A)\neq \A$. By Lemma \ref{Pre_Case_1}, $\A = \A_{1} \cup \A_{2}$, and
either  \textnormal{(\ref{poss1})} or \textnormal{(\ref{poss2})} holds. 
Note that in either case $S_{aj}(\A)=\A$ and $S_{ja}(\A)=\A$ for all $j\in [n]\setminus \{a,b\}$. Note also that $S_{ab}(\A)=S_{ba}(\A)$. 

We have two cases. 

\noindent
{\bf Case 1.} Suppose (\ref{poss1}) holds. Then $S_{bj}(\A)=\A$ and $S_{jb}(\A)=\A$ for all $j\in [n]\setminus \{a,b\}$. Suppose there exist $k,l \in [n]$ with $S_{kl}(\A)\neq \A$ and $k,l\neq a,b$. Again by Lemma \ref{Pre_Case_1}, $\A = \A_{3} \cup \A_{4}$ where 
$\A_{3} \subseteq \{C \in \B(n): \{k\} \in
C\}$ and $\A_{4} \subseteq \{C \in \B(n): \{k, l\}\in C\}$.
Therefore
\[ \A = \D_{1} \cup \D_{2}\cup \D_3\cup \D_4,  \]
where 
\begin{align}
 \D_{1}  &\subseteq \{C \in \B(n): \{a\}, \{b\}, \{k\} \in
C\}, \notag\\
\D_{2} & \subseteq \{C \in \B(n): \{a\}, \{b\}, \{k,l\} \in
C\}, \notag\\
 \D_{3} & \subseteq \{C \in \B(n): \{a,b\}, \{k\} \in
C\}, \notag\\
 \D_{4} & \subseteq \{C \in \B(n): \{a,b\}, \{k,l\} \in
C\}. \notag
\end{align}
Now $k\neq y_3,\dots, y_t$, for $S_{kl}(\A)\neq \A$. Therefore $\vert \A\vert\leq 4B_{n-t-1}< B_{n-t}-\tilde{B}_{n-t}-\tilde{B}_{n-t-1}$ (Lemma \ref{less}), a contradiction, as $\mathcal{A}$ is a maximum size non-trivial $t$-intersecting family. 

So we may assume that $S_{ij}(\A)=\A$ for all $i, j
\in [n]$ with $(i,j)\neq (a,b), (b,a)$. We first show that the interesting property of
$\A$ can be partially transferred to the family
$\sigma(\A)$ of sets which are union of singletons. In
particular, we show the following {\em cross-intersecting}
property:
\begin{eqnarray}
\vert \sigma(P) \cap \sigma(R) \cap [n] \setminus \{a,b, y_3,\dots, y_t\} \vert\geq 2, ~~~\forall P \in \A_{1}, R \in \A_{2}.
\label{e4}
\end{eqnarray}
Assume for a contradiction that there exist $P \in \A_{1}$ and
$R \in \A_{2}$ such that 
\begin{equation}
\vert \sigma(P) \cap \sigma(R) \cap [n]
\setminus \{a,b, y_3,\dots, y_t\}\vert \leq 1.\notag
\end{equation}
Since $P$ contains $\{a\}, \{b\}, \{y_3\},\dots, \{y_t\}$ and $R$ contains $\{a,b\}, \{y_3\},\dots, \{y_t\}$, we conclude that $P$ and $R$ must have at least one block of size at least $2$ in common.
Suppose there are $s \ge 1$ such common blocks of $P$ and $R$, say
$C_{1}$, $\ldots$, $C_{s}$, which are disjoint from $\sigma(P)
\cup \sigma(R) \cup \{a,b,y_3,\dots, y_t\}$. Fix two distinct points $w_{i}$,
$z_{i}$ from each block $C_{i}$. Then, since $S_{ij}(\A) = \A$ for all $i, j \in [n]$ with $(i,j)\neq (a,b), (b,a)$, we have
\[ R^{*} = s_{w_{s},z_{s}}( \cdots (s_{w_{1},z_{1}}(R)) \cdots )
\in \A. \]
However, $\vert P\cap R^{*}\vert \leq t-1$,
contradicting the $t$-intersecting property of $\A$. This
proves (\ref{e4}). 

Note that $\{a\}, \{b\}, \{y_3\}, \dots, \{y_t\}\in P$ for all $P\in S_{ab}(\A)$. This implies that $s_{ab}(P)\notin \A$ for all $P\in \A_2$.  Furthermore if $P\in \A_2$, then by (\ref{e4}), $\vert \sigma(P)\cap [n]\setminus \{a,b, y_3,\dots, y_t\}\vert \geq 2$. Therefore $\vert \sigma(s_{ab}(P))\vert \geq t+2$ for all $P\in \A_2$. Similarly, $\vert \sigma(P)\vert \geq t+2$ for all $P\in \A_1$. 
Therefore $\vert \A\vert\leq  B_{n-t}-\tilde{B}_{n-t}-(n-t)\tilde{B}_{n-t-1}<B_{n-t}-\tilde{B}_{n-t}-\tilde{B}_{n-t-1}$, a contradiction, as $\mathcal{A}$ is a maximum size non-trivial $t$-intersecting family. 

\noindent
{\bf Case 2.} Suppose (\ref{poss2}) holds. Suppose there exist $k,l \in [n]$ with $S_{kl}(\A)\neq \A$, $k\neq a$, and $(k,l)\neq (b,a)$. Again by Lemma \ref{Pre_Case_1}, we deduce that
\[ \A = \D_{5} \cup \D_{6}\cup \A_2,  \]
where 
\begin{align}
 \D_{5}  &\subseteq \{C \in \B(n): \{a\},  \{k\} \in
C\}, \notag\\
\D_{6} & \subseteq \{C \in \B(n): \{a\},  \{k,l\} \in
C\}. \notag
\end{align}
Now $k\neq x_2,\dots, x_t$, for $S_{kl}(\A)\neq \A$. Therefore $\vert \A\vert\leq 3B_{n-t-1}< B_{n-t}-\tilde{B}_{n-t}-\tilde{B}_{n-t-1}$ (Lemma \ref{less}), a contradiction.

So we may assume that $S_{ij}(\A)=\A$ for all $i, j
\in [n]$ with $(i,j)\neq (a,b), (b,a)$. As in the proof of (\ref{e4}), we can show that the following cross-intersecting
property holds:
\begin{eqnarray}
\vert \sigma(P) \cap \sigma(R) \cap [n] \setminus \{a,b,x_2,\dots, x_t\} \vert\geq 
1, ~~~\forall P \in \A_{1}, R \in \A_{2}.
\label{ee4}
\end{eqnarray}

Suppose $\A_2$ contains a $P_1$ with $\vert \sigma(P_1)\vert=t$.  Let $\sigma(P_1)=\{x_2,\dots, x_t, y\}$. Note that $y\in [n]\setminus \{a,b, x_2,\dots, x_t\}$. By  (\ref{ee4}), every element in $\A_1$ contains the singletons $\{a\}$, $\{x_2\}$, $\dots$, $\{x_t\}$, and $\{y\}$. Therefore $\vert \A_1\vert\leq B_{n-t-1}$, $\vert \A_2\vert\leq B_{n-t-1}$ and $\vert \A\vert\leq 2B_{n-t-1}<B_{n-t}-\tilde{B}_{n-t}-\tilde{B}_{n-t-1}$ (Lemma \ref{less}), a contradiction. So we may assume that $\A_2$ does not contain any $P$ with $\vert \sigma(P)\vert=t$.

Note that by (\ref{ee4}), $\vert\sigma(P)\vert\geq t+1$ for all $P\in\A_1$. So there are two subcases to be considered.

\noindent
{\bf Subcase 2.1.} Suppose $\A_1$ contains a $P$ with $\vert\sigma(P)\vert=t+1$. Let $P_1,\dots, P_r$ be the only  elements in $\sigma(\A)$ with $\vert\sigma(P_i)\vert=t+1$. Let $\sigma(P_i)=\{a,x_2,\dots, x_t,z_i\}$. Note that by (\ref{ee4}), $z_i\neq b$. Furthermore $r\leq n-t-1$. If $r=n-t-1$, then by (\ref{ee4}), $\A_2=\{Q(a,b)\}$. 
If $t=1$, then we conclude that $\A=\mathcal{H}(a,b)$, as $\A$ is a non-trivial $1$-intersecting family of maximum size, but this contradicts that $\A$ is not compressed. If $t>1$, then $Q(x_2,b)\notin \A$ since $Q(x_2,b)$ is not of the form given in (\ref{poss2}). But $\A\cup \{Q(x_2,b)\}$ is $t$-intersecting,  contradicting the fact that $\A$ is a non-trivial $t$-intersecting family of maximum size. Similarly, $r\neq n-t-2$. So $r\leq n-t-3$. 

Note that if $P\in\A_1$, then  $\sigma(P)\neq \{a,x_2,\dots, x_t\}$ and $\sigma(P)\neq \{a, x_2,\dots, x_t, v\}$ for $v\in [n]\setminus \{a,b,x_2,\dots, x_t, z_1,\dots, z_r\}$. So 
\[ \vert\A_1\vert\leq B_{n-t}-\tilde{B}_{n-t}-(n-t-r-1)\tilde{B}_{n-t-1}.  \]
Now if $P\in\A_2$, then by (\ref{ee4}), $\sigma(P) \supseteq \{x_2,\dots, x_t,z_1,\dots, z_r\}$. So $\vert \A_2\vert\leq B_{n-1-t-r}$.
Assume for the moment that $r\geq 2$. Then $\vert \A_2\vert\leq B_{n-t-3}<\tilde{B}_{n-t-1}$ (by (\ref{inequal})), and
\begin{align}
 \vert\A\vert &\leq B_{n-t}-\tilde{B}_{n-t}-(n-t-r-2)\tilde{B}_{n-t-1}\notag\\
 & \leq B_{n-t}-\tilde{B}_{n-t}-(n-t-2)\tilde{B}_{n-t-1}+(n-t-3)\tilde{B}_{n-t-1}\notag\\
 & = B_{n-t}-\tilde{B}_{n-t}-\tilde{B}_{n-t-1},\notag
\end{align}
a contradiction. 

Suppose $r=1$. Then by (\ref{ee4}), every element in $\A_2$ contains the singletons $\{x_2\}$, $\dots$, $\{x_t\}$, $\{z_1\}$. Since $\A_2$ does not contain any $P$ with $\vert\sigma(P)\vert=t$, we have
 $\vert \A_2\vert\leq B_{n-t-2}-\tilde{B}_{n-t-2}=\tilde{B}_{n-t-1}$ (by (\ref{ttt})), and 
\begin{align}
 \vert\A\vert &\leq B_{n-t}-\tilde{B}_{n-t}-(n-t-2)\tilde{B}_{n-t-1}+\tilde {B}_{n-t-1}\notag\\
  & < B_{n-t}-\tilde{B}_{n-t}-\tilde{B}_{n-t-1},\notag
\end{align}
a contradiction.

\noindent
{\bf Subcase 2.2.} Suppose $\vert\sigma (P)\vert\geq t+2$ for all $P\in\A_1$. Then
\begin{equation}
\vert \A_1\vert\leq B_{n-t}-\tilde{B}_{n-t}-(n-t)\tilde{B}_{n-t-1}.\notag
\end{equation}

By (\ref{ee4}), every $P\in \A_2$ must contain a singleton distinct from $\{a\}$, $\{b\}$, $\{x_2\}$, $\dots$, $\{x_t\}$. Since $\A_2$ does not contain any $P$ with $\vert\sigma(P)\vert=t$, we have
 $\vert \A_2\vert\leq (n-t-1)(B_{n-t-2}-\tilde{B}_{n-t-2})=(n-t-1)\tilde{B}_{n-t-1}$ (by (\ref{ttt})), and 
\begin{align}
 \vert\A\vert &\leq B_{n-t}-\tilde{B}_{n-t}-(n-t)\tilde{B}_{n-t-1}+(n-t-1)\tilde {B}_{n-t-1}\notag\\
  & = B_{n-t}-\tilde{B}_{n-t}-\tilde{B}_{n-t-1},\notag
\end{align}
a contradiction.  This completes the proof of the theorem.
\end{proof}

\begin{thm}\label{Case_2}
Let $\mathcal{A}$ be a non-trivial $t$-intersecting family of set
partitions of $[n]$ of maximum size. Suppose $\sigma(\A)$ is a
non-trivial $t$-intersecting family of subsets of $[n]$. Then, for
$n\geq n_0(t)$,
\[ |\A| = B_{n-t}-\tilde{B}_{n-t}-\tilde{B}_{n-t-1}+t.\]
Moreover, $\mathcal{A} = \mathcal{H}(a_1,\dots, a_t,b)$ for some $a_1,\dots, a_t, b \in
[n]$.
\end{thm}

\begin{proof}
For $k \ge t+1$, let $\F_{k} = \sigma(\A) \cap {[n] \choose k}$.
Since $\sigma(\A)$ is $t$-intersecting, by applying the Erd{\H
o}s-Ko-Rado theorem to $\F_{k}$ for each $k \le \lfloor \frac{n}{t+1}+t-1
\rfloor$, we have
\begin{eqnarray}
|\A| \le \sum_{k=t+1}^{\lfloor \frac{n}{t+1}+t-1
\rfloor} {n-t \choose
k-t}\tilde{B}_{n-k} + \sum_{k=\lfloor \frac{n}{t+1}+t-1
\rfloor + 1}^{n} {n
\choose k} \tilde{B}_{n-k}. \label{e5}
\end{eqnarray}

We consider the following cases.

\noindent {\bf Case 1.} $\F_{t+1} = \emptyset$.

Then the sum in (\ref{e5}) starts from $k=t+2$, and by  (\ref{e1}) and Lemma \ref{less03}:

\begin{eqnarray}
|\A| & \le & \sum_{k=t+2}^{\lfloor \frac{n}{t+1}+t-1
\rfloor} {n-t \choose
k-t}\tilde{B}_{n-k} + \sum_{k=\lfloor \frac{n}{t+1}+t-1
\rfloor + 1}^{n} {n
\choose k} \tilde{B}_{n-k} \notag\\
& < & \sum_{k=t+2}^{n} {n-t \choose
k-t}\tilde{B}_{n-k} + \tilde{B}_{n-t-1}\notag
\\
& = & B_{n-t}-{n-t \choose 0}\tilde{B}_{n-t}-{n-t \choose
1}\tilde{B}_{n-t-1} + \tilde{B}_{n-t-1} \notag\\
& = &
B_{n-t}-\tilde{B}_{n-t}-\tilde{B}_{n-t-1}-(n-t-2)\tilde{B}_{n-t-1}\notag
\\
& < & B_{n-t}-\tilde{B}_{n-t}-\tilde{B}_{n-t-1}\notag
\end{eqnarray}
for sufficiently large $n$. This contradicts the maximality of $\A$.

\noindent {\bf Case 2.} $\F_{t+1} \not = \emptyset$.

\noindent {\bf Subcase 2.1.} $\vert \cap_{F \in \F_{t+1}} F \vert <t$.

Then there exist three sets $F_1, F_2$, $F_3 \in \mathcal{F}_{t+1}$ such that $F_1\cap F_2\nsubseteq F_3$. Note that $F_3$ must contain the symmetric difference $F_1\Delta F_2$, and since $\vert F_3\cap F_i\vert\geq t$ for $i=1,2$, $F_3$ must take the form $(F_1\cup F_2)\setminus \{x\}$ for some $x\in F_1\cap F_2$. Indeed, all sets in $\F_{t+1}$ other than $F_1$ and $F_2$ must also have this form.

Let
\begin{align}
A_0 &= \{1,2,\dots,t,t+1\},\notag\\
A_1 &= \{1,2,\dots, t,t+1,t+2\}\setminus \{1\},\notag\\
A_2 &= \{1,2,\dots, t,t+1,t+2\}\setminus \{2\},\notag\\
&\vdots\notag\\
A_{t+1} &= \{1,2,\dots, t,t+1,t+2\}\setminus \{t+1\}.\notag
\end{align}
Without loss of generality, we may assume that $A_0,A_1,A_2\in \F_{t+1}$ and $\F_{t+1}\subseteq \{ A_0,A_1,A_2,\dots, A_{t+1}\}$.  In
view of the $t$-intersecting property of $\sigma(\A)$, if $P\in \A$ and $i\notin \sigma(P)$ for some $1\leq i\leq t+1$, then $A_i\subseteq \sigma(P)$, for  $A_0,A_1,A_2\in \F_{t+1}$. Hence for any $P \in \A$, $A_i \subseteq \sigma(P)$ for some $0\leq i\leq t+1$.
Now for sufficiently large $n$ (Lemma \ref{less}),
\[ |\A| \le (t+2) B_{n-t-1} < B_{n-t}-\tilde{B}_{n-t}-\tilde{B}_{n-t-1}, \]
contradicting the maximality of $\A$.

\noindent {\bf Subcase 2.2.} $\vert \cap_{F \in \F_{t+1}} F \vert=t$.

Without loss of generality, there exists $r \ge t+1$ such that
\[ \F_{t+1} = \{\{1,2,\dots, t, i\}: t+1 \le i \le r\}\]
for some $r \in \{t+1, \ldots,n\}$. Notice that $r \le n-1$;
otherwise, all the set partitions in $\A$ will contain $\{1\}$, $\{2\}$, $\dots$, $\{t\}$, 
contradicting the non-triviality of $\sigma(\A)$.

 Let $P \in \A$. Then either
$\{1,2,\dots, t\}\subseteq \sigma(P)$, or there is a $j\in \{1,2,\dots, t\}$ with $j \not \in \sigma(P)$ and $(\{1,2,\dots, t\}\setminus \{j\})\cup \{t+1,\dots, r\}\subseteq \sigma(P)$ (since $\sigma(P)$ must intersect every element in $\F_{t+1}$). In the former, we cannot have $\sigma(P)  = \{1,2,\dots, t\}$ or $\sigma(P) = \{1,2,\dots, t,x\}$ for all $x \in [n] \setminus \{1,2 \ldots, t, t+1,\dots, r\}$; in the later, $(\{1,2,\dots, t\}\setminus \{j\})\cup \{t+1,\dots, r\}\subseteq \sigma(P)$ where $j$ can take at most $t$ values. So if $t+4 \le r \le n-2$, then 

\begin{eqnarray}
|\A| & \le & B_{n-t}-\tilde{B}_{n-t}-{n-r \choose
1}\tilde{B}_{n-t-1} + tB_{n-r+1}\notag \\
& < & B_{n-t}-\tilde{B}_{n-t}-\tilde{B}_{n-t-1}
~~~(\textrm{Lemma}~ \ref{less02}).\notag
\end{eqnarray}

Suppose  $t+2 \le r \le t+3$. Assume $\{1,2,\dots, t\}\subseteq \sigma(P)$. The number of $P\in \A$ with $\{1,2,\dots, t,i\}\subseteq \sigma(P)$ ($t+1\leq i\leq r$) is at most $3B_{n-t-1}$. The number of $P\in \A$ with $\{1,2,\dots, t,i\}\nsubseteq \sigma(P)$ for all $i=t+1,t+2,\dots, r$, is at most $\sum_{k=2}^{n-r} {n-r \choose k} \tilde B_{n-r-k}< B_{n-r}<B_{n-t-1}$.  Therefore for sufficiently large $n$ (Lemma \ref{less}),
\[ |\A| \le 3B_{n-t-1} + B_{n-t-1} +  tB_{n-r+1}\leq (t+4)B_{n-t-1}<
B_{n-t}-\tilde{B}_{n-t}-\tilde{B}_{n-t-1}. \]

Suppose $r=t+1$ i.e. $\F_{t+1}=\{\{1,2,\dots, t, t+1\}\}$.  As in Case 1, for sufficiently large $n$,
\begin{eqnarray}
|\A| & \le & \tilde{B}_{n-t-1} + \sum_{k=t+2}^{\lfloor \frac{n}{t+1}+t-1
\rfloor} {n-t \choose
k-t}\tilde{B}_{n-k} + \sum_{k=\lfloor \frac{n}{t+1}+t-1
\rfloor+ 1}^{n} {n
\choose k} \tilde{B}_{n-k}\notag \\
& < & B_{n-t}-\tilde{B}_{n-t}-\tilde{B}_{n-t-1}-(n-t-3)\tilde{B}_{n-t-1} \notag
\\
& < & B_{n-1}-\tilde{B}_{n-1}-\tilde{B}_{n-2}.\notag
\end{eqnarray}

Hence, $r=n-1$ and $\A = \mathcal{H}(1,2,\dots, t, n)$.

\end{proof}

\noindent {\bf Proof of Theorem \ref{main}}.

Let $\A$ be a non-trivial $t$-intersecting family of maximum size.
Repeatedly apply the splitting operations until we obtain a family $\A^{*}$ such that $\A^{*}$ is compressed (Proposition \ref{P2}). Note that by Theorem \ref{Case_1}, we may choose the splitting operations so that $\A^{*}$
is non-trivially $t$-intersecting. Therefore $\sigma(\A^{*})$ is non-trivially $t$-intersecting (for $\sigma(\A^{*})$ is $t$-intersecting by Proposition \ref{P3}), and the result follows from Theorem \ref{Case_2} and Proposition \ref{undo}.

\end{document}